\documentclass[10pt,a4paper]{article}
\usepackage{amssymb} 
\usepackage{amsmath} 
\usepackage{amsthm} 

\newtheorem{theorem}{Theorem}
\newtheorem{proposition}[theorem]{Proposition}

\newtheorem*{markov}{Markov's inequality}
\newtheorem*{chernoff}{Chernoff's bound}

\newcommand{\Exp}{\,\mathbb{E}}
\renewcommand{\Pr}{\,\mathbb{P}}

\DeclareMathOperator{\Bin}{Bin}

\DeclareMathOperator{\ch}{ch}

\title{Improper choosability and {P}roperty {B}}

\author{
Ross J. Kang\thanks{This work was completed while the author was at Durham University, supported by EPSRC, grant EP/G066604/1. He is currently supported by a NWO Veni grant.} \\
Centrum Wiskunde \& Informatica \\
PO Box 94079, 1090 GB \\
Amsterdam, The Netherlands \\
{\tt ross.kang@gmail.com}
}

\begin{document}

\maketitle

\begin{abstract}
A fundamental connection between list vertex colourings of graphs and Property B (also known as hypergraph 2-colourability) was already known to Erd\H{o}s, Rubin and Taylor~\cite{ERT80}.  In this article, we draw similar connections for improper list colourings.  This extends results of Kostochka~\cite{Kos02}, Alon~\cite{Alo92,Alo93,Alo00}, and Kr\'al' and Sgall~\cite{KrSg05} for, respectively, multipartite graphs, graphs of large minimum degree, and list assignments with bounded list union.
\end{abstract}

\section{Introduction and background}
\label{sec:intro}

This work is an attempt to galvanise investigations into improper choosability outside the setting of topologically embedded graph families.  We focus on extending the well-known relationship between choosability and Property B.  This compels us to consider improper choosability for ($t$-improperly) multipartite graphs, minimum degree graphs and list assignments whose list union is bounded.  Our results depend on probabilistic methods.

Let $G = (V,E)$ be a simple, undirected graph.
We refer to a {\em colouring} of $G$ as being an arbitrary mapping $c: V\to \mathbb{Z}^+$.  If $|c(V)| = k$ for some integer $k$, then we call $c$ a {\em $k$-colouring}.
A colouring $c$ of $G$ is {\em proper} if $c(u)\ne c(v)$ for any $uv \in E$; equivalently, $c$ is proper if $c^{-1}(c(V))$ is a partition of $V$ such that each part induces a subgraph with maximum degree $0$.
The latter definition of {\em proper} directly generalises as follows.  For a non-negative integer $t$, a colouring $c$ of $G$ is {\em $t$-improper} if $c^{-1}(c(V))$ yields a partition of $V$ such that each part induces a subgraph of maximum degree at most $t$; in other words, $c$ is $t$-improper if for every vertex $u$ the number of neighbours $v$ of $u$ such that $c(u)=c(v)$ is at most $t$.  We say $G$ is {\em $t$-improperly $k$-colourable} if there exists a $t$-improper $k$-colouring of $G$.  The {\em $t$-improper chromatic number $\chi^t(G)$} of $G$ is the least $k$ such that $G$ is $t$-improperly $k$-colourable.
This parameter captures the usual {\em chromatic number $\chi(G)$} of $G$ as a special case, $t=0$.

This natural form of generalised graph colouring has been developed in many contexts.  Its provenance extends at least as far back as the 1980's, when it was independently introduced by several authors~\cite{AnJa85,CCW86,HoSt86}; indeed, it may well be argued that Lov\'asz~\cite{Lov66} (in the mid-1960's, while he was a teenager) was the first to study this type of vertex partition.  In the literature, it is sometimes referred to as {\em defective colouring}, {\em $(k,t)$-colouring} or {\em relaxed colouring}. 

Most of our attention is restricted to the stronger {\em list} form of improper colouring --- that is, the form of graph colouring in which an adversary may place individual restrictions on the colours used at each vertex.  Let us establish a few more definitions before continuing with the background.

We refer to $[\ell] = \{1,\dots,\ell\}$ as the {\em spectrum of colours}, where $\ell$ is some positive integer.  Given a positive integer $k\le \ell$, a mapping $L: V\to \binom{[\ell]}{k}$ is called a {\em $(k,\ell)$-list-assignment} of $G$; a colouring $c$ of $G$ is called an {\em $L$-colouring} if $c(v)\in L(v)$ for any $v\in V$.  For a non-negative integer $t$, we say $G$ is {\em $t$-improperly $(k,\ell)$-choosable} if for any $(k,\ell)$-list-assignment $L$ of $G$ there is a $t$-improper $L$-colouring of $G$.  We say $G$ is {\em $t$-improperly $k$-choosable} if it is $t$-improperly $(k,\ell)$-choosable for any $\ell\ge k$.  The {\em $t$-improper choosability $\ch^t(G)$} (or {\em $t$-improper choice number} or {\em $t$-improper list chromatic number}) of $G$ is the least $k$ such that $G$ is $t$-improperly $k$-choosable.
Note $G$ is $t$-improperly $k$-colourable iff it is $t$-improperly $(k,k)$-choosable; thus, $\ch^t(G)\ge \chi^t(G)$.  Note also $\ch^t(G)\le\ch(G)$.

In the case $t = 0$, we omit the modifier `$0$-improper(ly)' and occasionally substitute `proper(ly)'.  Proper list colouring was introduced independently by Vizing~\cite{Viz76} and Erd\H{o}s, Rubin and Taylor~\cite{ERT80} and has been subject to the substantial scrutiny of many researchers in discrete mathematics and theoretical computer science.  The choice number has been approached from various angles including topological~\cite{Tho94}, algebraic~\cite{Alo99} and probabilistic~\cite{MoRe02}.
In this paper, we take the last-mentioned tack, in particular to generalise results of Kostochka~\cite{Kos02}, Alon~\cite{Alo92,Alo93,Alo00}, and Kr\'al' and Sgall~\cite{KrSg05}.  (Our definition of $t$-improper choosability, with a bounded spectrum, follows the formulation given in~\cite{KrSg05}.)

For $t \ge 1$, the $t$-improper list colouring problem has been studied thoroughly, though mainly within the realm of planar graphs and related graph classes.  The investigation of $t$-improper list colouring was launched by Eaton and Hull~\cite{EaHu99} and \v{S}krekovski~\cite{Skr99a}, who independently showed every planar graph to be $2$-improperly $3$-choosable, a result inspired by the remarkable proof of Thomassen~\cite{Tho94} that every planar graph is $5$-choosable.  Recently, Cushing and Kierstead~\cite{CuKi10} completed the stubborn last step in the program of $t$-improperly list colouring planar graphs, by proving every planar graph is $1$-improperly $4$-choosable.  (Voigt~\cite{Voi93} exhibited a planar graph that is not properly $4$-choosable and there are elementary examples of planar graphs that are not $1$-improperly $3$-colourable and not $k$-improperly $2$-colourable for any fixed positive $k$~\cite{CCW86}.)
There have continued to be active efforts on $t$-improper list colouring in classes related to graphs on surfaces (including the plane), often with added girth or density conditions~\cite{CZW08,DoXu09a,HaSe06,LSWZ01,Mia03,Skr99b,Skr00,Woo01,Woo02,Woo04,Woo11+,WoWo09,XuYu08,XuZh07}; many of these efforts have employed discharging techniques.
We comment that some results on improper choosability are obtained within the scope of {\em generalised list colouring}~\cite{BBM97,BDM95}, in which analogues of Brooks' and Gallai's theorems were established.

For researchers in this area, a folklorish question has been circulating for some time, at least as far back as 2004, from personal recollection.  It asks about what is the best general lower bound for the $t$-improper choosability in terms of the choosability.  (An easy argument shows that the $t$-improper chromatic number is at least a $(t+1)$th fraction of the chromatic number; this is tight.)  A corollary to our methods here is a first step for this informal problem, even if the question remains open --- see the remarks after Theorem~\ref{thm:degeneracy}.

The structure of the paper is as follows.  In Section~\ref{sec:results}, we outline and discuss our main results in detail.  In Section~\ref{sec:propB}, we review Property B, then draw a correspondence between an extremal study of families with Property B and the $t$-improper choosability of $t$-improperly colourable graphs.  In Section~\ref{sec:mindeg}, we study the relationship between the minimum degree of a graph and $t$-improper choosability.  The proofs of the main results are found in Section~\ref{sec:proofs} and then some concluding remarks are given in Section~\ref{sec:conclusion}.  Before continuing, we fix more of our notation.

\subsection{Notation}

We follow standard notation in graph theory and any concepts or notation not defined here may be found in a standard text, e.g.~\cite{BoMu08} or~\cite{Die05}.
For a graph $G$, we use $\delta(G)$, $\delta^*(G)$ and $\Delta(G)$ to denote, respectively, the minimum degree, degeneracy and  maximum degree of $G$.
We use the shorthand $[\ell] = \{1,\dots,\ell\}$.
By {\em $k$-set} or {\em $k$-subset}, we mean one with $k$ elements.

We refer to two basic probabilistic tools, cf.~\cite{MoRe02}.

\begin{markov}
For any positive random variable $X$,
\begin{align*}
\Pr(X\ge t)\le \Exp(X)/t.
\end{align*}
\end{markov}

\begin{chernoff}
For any $0\le t\le np$,
\begin{align*}
\Pr(|\Bin(n,p)-np|> t)\le 2\exp(-t^2/(3np)).
\end{align*}
\end{chernoff}

\section{Overview of the main results}
\label{sec:results}

Our main contribution is to closely tie improper choosability problems to the extremal study of Property B for families of $k$-sets.  However, we do not introduce Property B until the next section.  As we shall note in Section~\ref{sec:proofs}, our main theorems are corollary to the more detailed results given in Sections~\ref{sec:propB} and~\ref{sec:mindeg}, in combination with known bounds for the extremal behaviour of Property B.

Let us define $\mathcal{K}_t(n*j)$ to be the class of graphs $G=(V,E)$ that admit a partition $(V_1,\dots,V_j)$ of $V$ such that, for all $i,i'$ with $i\ne i'$, $|V_i|=n$, $V_i\times V_{i'}\subseteq E$ and $V_i$ induces a subgraph of maximum degree at most $t$.  Informally, we refer to $\mathcal{K}_t(n*j)$ as the class of {\em complete $t$-improperly $j$-partite graphs}.
We have determined the asymptotic behaviour of the $t$-improper choice numbers of graphs in this class.

\begin{theorem}
\label{thm:completemultipartite}
For fixed $t\ge0,j\ge 2$, for any $K \in \mathcal{K}_t(n*j)$, as $n \to \infty$,
\begin{align*}
\ch^t(K) = (1 + o(1))\frac{\ln n}{\ln \left(j/(j-1)\right)}.
\end{align*}
\end{theorem}

\noindent
In the case $t = 0$, this result was first shown by Alon~\cite{Alo92}.
With a more involved probabilistic analysis, Gazit and Krivelevich~\cite{GaKr06} have obtained a detailed description that allows for distinct part sizes; however, for $t\ge 1$ we have left this to future study.

Recall that a bound on the degeneracy $\delta^*(G)$ of a graph $G$ implies a bound on the choosability $\ch(G)$ of $G$: it is clear that $\ch(G) \le \delta^*(G)+1$ by simply noting that the graph can be coloured greedily, always picking a vertex of degree at most $\delta^*(G)$ as the next vertex to colour.  Obviously then, $\ch^t(G)\le \delta^*(G)+1$ for any $t$.  In the next theorem, we have obtained a lower bound on the $t$-improper choosability in terms of the degeneracy. (Note that the minimum degree is no smaller than the degeneracy.)
In particular, it follows that any graph has bounded $t$-improper choosability iff it has bounded degeneracy.  This was first established in the case $t=0$ by Alon~\cite{Alo93}, later tightened~\cite{Alo00}.

\begin{theorem}
\label{thm:degeneracy}
For fixed $t\ge0$, there is a function $j = j(d)$ that satisfies $j(d) = (1/2 + o(1))\log_2 d$ as $d\to \infty$ such that, if $\delta(G)\ge d$, then
$\ch^t(G)\ge j(d)$.
\end{theorem}

\noindent
By Theorem~\ref{thm:completemultipartite} for $j=2$, this lower bound is asymptotically correct up to a factor of $2$ (where the asymptotics are in terms of $d$).  Incidentally, the above remarks imply that $\ch^t(G)$ is at least a fixed multiple of $\ln (\ch(G))$.  Since each colour class of a $t$-improper colouring may be greedily properly coloured with at most $t+1$ distinct colours, we see that $\chi^t(G)\ge \chi(G)/(t+1)$ always. It remains unclear if a similar relationship between $\ch^t$ and $\ch$ should hold, and this folklorish question has been discussed at least as far back as 2004.

Kr\'al' and Sgall~\cite{KrSg05} were the first to explicitly consider bounded spectrum list colourings and were able to handle the following questions:
\begin{enumerate}
\item Given a number $k$, does there exist a number $\ell_k$ such that a graph G is properly $k$-choosable if it is properly $(k,\ell_k)$-choosable?
\item Is it true that for some $k\ge3$ and $\ell\ge k$ there exists a number $K_{k,\ell}$ such that each properly $(k,\ell)$-choosable graph G is properly $K_{k,\ell}$-choosable?
\end{enumerate}
They showed that the answer to the first question is mostly `no'.  For the second question, they demonstrated that $K_{k,\ell}$ does not exist if $k \le 2k-2$ and that $K_{k,\ell} = O(2^{4k}k\ln k)$ if $k \ge 2k-1$.  For the second question, we obtain an analogous result for $t$-improper colouring.

\begin{theorem}
\label{thm:boundedspectrum}
Let $t$ be a non-negative integer and let $G$ be a graph.
\begin{enumerate}
\item
If $\chi^t(G)\le 2$, then 
$G$ is $t$-improperly $(k,2k-2)$-choosable.
\item
If $G$ is $t$-improperly $(k,2k-1)$-choosable, then 
$\ch(G)= O(2^{4k}k t\ln k)$.
\end{enumerate}
\end{theorem}

\noindent
Regarding the first question, it is largely left open to show the existence of graphs that are $t$-improperly $(k,\ell)$-choosable but not $t$-improperly $(k,\ell+1)$-choosable for all $3\le k\le \ell$ --- Kr\'al' and Sgall showed this in the case $t=0$.  The present research was partly motivated by a related question on a recently-launched theoretical computer science online research forum~\cite{Epp10}, in which it was asked whether a singly-exponential-time test for proper $k$-choosability can be obtained.  This question is still open for improper $k$-choosability.

\section{Property B and complete multipartite graphs}
\label{sec:propB}

In this section, we draw the connection between Property B and improper list colouring of complete multipartite graphs.  We begin by considering bipartite graphs, then extend to multipartite graphs.

Following Erd\H{o}s, Rubin and Taylor~\cite{ERT80}, we derive a good estimate of the $t$-improper choice number of $t$-improperly $2$-colourable graphs
by relating the problem to the extremal analysis of Property B~\cite{Erd63,Erd64,Erd69,ErHa61}.
A family $\mathcal{F}$ of sets has {\em Property B} if there exists a set $B$ which meets every set in $\mathcal{F}$ but contains no set in $\mathcal{F}$.
Property B for a family of $k$-sets is equivalent to proper $2$-colourability of $k$-uniform hypergraphs.

For a fixed integer $k\ge2$, let $M(k)$ be the cardinality of a smallest family of $k$-sets that does not have Property B.
For fixed integers $k,\ell\ge 2$, let $M(k,\ell)$ be the cardinality of a smallest family of $k$-subsets of $[\ell]$ that does not have Property B.  Note that $M(k,2k-1)=\binom{2k-1}{k}$ since the collection $\binom{[2k-1]}{k}$ of all $k$-subsets of $[2k-1]$ does not have Property B, whereas any proper subcollection of $\binom{[2k-1]}{k}$ has Property B.  It also holds that $M(k,\ell) = \infty$ if $\ell \le 2k-2$, as every subcollection of $\binom{[2k-2]}{k}$ has Property B.  Clearly, $M(k) = \inf_{\ell\ge 2k-1}M(k,\ell)$.

The best general upper bound on $M(k)$ is a probabilistic construction due to Erd\H{o}s~\cite{Erd64}, while the best lower bound is due to an ingenious application of ``semirandom local recolouring'' by Radhakrishnan and Srinivasan~\cite{RaSr00}:
\begin{align}
\label{eqn:propB}
\Omega\left(\sqrt{\frac{k}{\ln k}} 2^k\right)
\le
M(k)
\le
O\left(k^2 2^k\right).
\end{align}

This next proposition generalises a theorem of Erd\H{o}s, Rubin and Taylor.  Similar results were obtained by Hull~\cite{Hul97}.

\begin{proposition}
\label{prop:bipartite,propB}
Let $k,\ell$ be integers such that $2\le k\le \ell$.
\begin{enumerate}
\item
There is a $(k,\ell)$-list-assignment $L$ of the complete properly bipartite graph $K_{n_1,n_2}$, with $n_1\ge M(k,\ell)$ and $n_2\ge (tk+1) M(k,\ell)$, such that there is no $t$-improper $L$-colouring.
\item
Given a non-negative integer $t$, if $G = (V,E)$ is a $t$-improperly $2$-colourable graph and $|V| < M(k,\ell)$, then $G$ is $t$-improperly $(k,\ell)$-choosable.
\end{enumerate}
\end{proposition}

\begin{proof}
The proofs are as follows.
\begin{enumerate}
\item
We define $L$ as follows.  Let $(V_1,V_2)$ be the bipartition of the vertices of $K_{n_1,n_2}$ with $|V_1|=n_1$ and $|V_2|=n_2$.  Let $\mathcal{F}$ be a smallest family of $k$-subsets of $[\ell]$ that does not have Property B. Assign a copy of each of the $M(k,\ell)$ $k$-sets in $\mathcal{F}$ to the vertices of $V_1$.  Assign $tk+1$ copies of each of the $k$-sets in $\mathcal{F}$ to the vertices of $V_2$.  Now, for a contradiction, assume that there exists a $t$-improper $L$-colouring $c$.  Necessarily, the set $\cup_{v\in V_1}\{c(v)\}$ meets every member of $\mathcal{F}$; therefore, since $\mathcal{F}$ does not have Property B, there is some $F\in \mathcal{F}$ such that $F \subseteq \cup_{v\in V_1}\{c(v)\}$.  Since there are $tk+1$ occurrences of $F$ as the list of some vertex in $V_2$, by the pigeonhole principle there must be some $f\in F$ such that $c(v)=f$ for at least $t+1$ vertices $v$ of $V_2$.  As every vertex in $V_1$ is adjacent to every vertex in $V_2$, this contradicts that $c$ is a $t$-improper colouring.
\item
Let $(V_1,V_2)$ be a partition of $V$ such that the induced subgraphs $G[V_1]$ and $G[V_2]$ each have maximum degree at most $t$.
Let $L$ be any $(k,\ell)$-list-assignment.  The family $\mathcal{F}=\{L(v)\}_{v\in V}$ is a family of $k$-subsets of $\ell$ of size less than $M(k,\ell)$, and so must have Property B.  Let $B$ be a set that witnesses Property B for $\mathcal{F}$, i.e.~$B$ meets every set in $\mathcal{F}$ but contains no set in $\mathcal{F}$.  We may use $B$ to define a colouring $c$ as follows: for $v\in V_1$, let $c(v)$ be chosen from $L(v) \cap B$; for $v\in V_2$, let $c(v)$ be chosen from $L(v)\setminus B$.  Since the vertices of $V_1$ are only coloured from $B$, the vertices of $V_2$ are only coloured from $[\ell]
\setminus B$, and $G[V_1]$, $G[V_2]$ each have maximum degree at most $t$, $c$ is a $t$-improper $L$-colouring.
\end{enumerate}
\end{proof}

We can extend Proposition~\ref{prop:bipartite,propB}, taking the strategy of Kostochka~\cite{Kos02}, by considering $t$-improperly multipartite graphs and by using an extension of Property B.  For $j \ge 2$, we say that a family $\mathcal{F}$ of sets has {\em Property B($j$)} if there exist mutually disjoint sets $B_1,\dots,B_j$ each of which meets every set in $\mathcal{F}$.  Clearly, Property B($2$) is equivalent to Property B.  Property B($j$) for a family of $k$-sets is equivalent to $j$-panchromaticity of $k$-uniform hypergraphs.

For fixed integers $j,k\ge2$, let $M(j,k)$ be the cardinality of a smallest family of $k$-sets that does not have Property B($j$).
For fixed integers $j,k,\ell$, with $2\le j,k\le \ell$, let $M(j,k,\ell)$ be the cardinality of a smallest family of $k$-subsets of $[\ell]$ that does not have Property B($j$).
Clearly, $M(j,k) = \inf_{\ell\ge j,k}M(j,k,\ell)$.

For $j,k\ge2$, good general bounds on $M(j,k)$ have recently been obtained by Shabanov~\cite{Sha10,Sha11}:
\begin{align}
\label{eqn:propBj}
\Omega\left(\frac1j\left(\frac{k}{j\ln k}\right)^{1/3}\left(\frac{j}{j-1}\right)^k\right)
\le
M(j,k)
\le
O\left(k^2\left(\frac{j}{j-1}\right)^k\ln j\right).
\end{align}
As for Property B, these bounds were obtained via probabilistic techniques.

\begin{theorem}
\label{thm:multipartite,propB}
Let $k,\ell$ be integers such that $2\le j, k\le \ell$.
\begin{enumerate}
\item
There is a $(k,\ell)$-list-assignment $L$ of the complete properly $j$-partite graph $K_{n_1,\dots,n_j}$, with $n_1 \ge M(j,k,\ell)$ and $n_i\ge (tk+1) M(j,k,\ell)$ for any $i\in\{2,\dots,j\}$, such that there is no $t$-improper $L$-colouring.
\item
Given a non-negative integer $t$, if $G = (V,E)$ is a $t$-improperly $j$-colourable graph and $|V| < M(j,k,\ell)$, then $G$ is $t$-improperly $(k,\ell)$-choosable.
\end{enumerate}
\end{theorem}

\begin{proof}
The proofs mirror those in the proof of Proposition~\ref{prop:bipartite,propB}.
\begin{enumerate}
\item
We define $L$ as follows.  Let $(V_1,\dots,V_j)$ be the $j$-partition of the vertices of $K_{n_1,\dots,n_j}$ with $|V_i|=n_i$ for any $i\in\{1,\dots,j\}$.  Let $\mathcal{F}$ be a smallest family of $k$-subsets of $[\ell]$ that does not have Property B($j$).  Assign a copy of each of the $M(j,k,\ell)$ $k$-sets in $\mathcal{F}$ to the vertices of $V_1$.  For each $i\in\{2,\dots,j\}$, assign $tk+1$ copies of each of the $k$-sets in $\mathcal{F}$ to the vertices of $V_i$.  Now, for a contradiction, assume that there exists a $t$-improper $L$-colouring $c$.
Obviously, the set $\cup_{v\in V_1}\{c(v)\}$ meets every member of $\mathcal{F}$.
For each $i\in\{2,\dots,j\}$, let $C_i$ be the set of colours $f$ such that there are at least $t+1$ vertices $v\in V_i$ such that $c(v)=f$.  We observe that $C_i$ meets every member of $\mathcal{F}$.  Otherwise, if $F$ is some $k$-set in $\mathcal{F}$ that does not intersect $C_i$, then, since $F$ occurs $tk+1$ times as the list of some vertex in $V_i$, it follows by the pigeonhole principle that there is some $f\in F$ such that $c(v)=f$ for at least $t+1$ vertices $v$ of $V_i$, a contradiction.
Since $K_{n_1,\dots,n_j}$ is a complete multipartite graph, the sets $\cup_{v\in V_1}\{c(v)\}$, $C_2,\dots,C_j$ are mutually disjoint, otherwise $c$ would not be $t$-improper.  But then these sets witness Property B($j$) for $\mathcal{F}$, which is a contradiction.
\item
Let $(V_1,\dots,V_j)$ be a partition of $V$ such that $\Delta(G[V_i])\le t$ for any $i\in\{1,\dots,j\}$.
Let $L$ be any $(k,\ell)$-list-assignment.  The family $\mathcal{F}=\{L(v)\}_{v\in V}$ is a family of $k$-subsets of $\ell$ of size less than $M(j,k,\ell)$, and so must have Property B($j$).  Let $B_1,\dots,B_j$ be the sets that witnesses Property B($j$) for $\mathcal{F}$, i.e.~they are mutually disjoint sets such that, for any $i\in\{1,\dots,j\}$, $B_i$ meets every set in $\mathcal{F}$.  We may use $B_1,\dots,B_j$ to define a colouring $c$ as follows: for any $i\in\{1,\dots,j\}$, for any $v\in V_i$, let $c(v)$ be chosen from $L(v) \cap B_i$.  Since for every $i\in\{1,\dots,j\}$ the vertices of $V_i$ are only coloured from $B_i$ and $G[V_i]$ has maximum degree at most $t$, $c$ is a $t$-improper $L$-colouring.
\end{enumerate}
\end{proof}

\section{A bound in terms of the minimum degree}
\label{sec:mindeg}

The following is a generalisation of a result of Kr\'al' and Sgall~\cite{KrSg05}.  We note that even in the special case $t=0$, this theorem extends their result by exposing a more direct connection with Property B for families of $k$-sets drawn from $[\ell]$.

\begin{theorem}
\label{thm:mindeg}
Fix integers $t,k,\ell$, with $t \ge 0$ and $2 \le k \le \ell$, and let
\begin{align*}
D = 12 (M(k,\ell))^2 \cdot \ln M(k,\ell) \cdot \ln k \cdot
\left(1+\sqrt{1+\frac{(tk+1)}{3\ln M(k,\ell)}}\right)^2.
\end{align*}
If $G = (V,E)$ is a graph with $\delta(G) \ge D$, 
then $G$ is not $t$-improperly $(k,\ell)$-choosable.
\end{theorem}

\begin{proof}
We assume, for a contradiction, that $G$ is $t$-improperly $(k,\ell)$-choosable.
The strategy of proof is the same as that introduced by Alon~\cite{Alo93}, subsequently refined~\cite{Alo00} and then adapted by Kr\'al' and Sgall~\cite{KrSg05}.  There are two stages of randomness.  In one stage, we choose a small random vertex subset $A$ and assign lists to the vertices of $A$ randomly.  After showing that, with positive probability, there are a large number of {\em good} vertices (to be defined below), we fix such a subset $A$ and an assignment of lists.  The good vertices guarantee in the second stage that, with positive probability, the remaining vertices cannot be coloured from random lists.  Without loss of generality, we assume that $D$ is integral.

Let $p = 1/(4M(k,\ell)\ln k)$.  Note that $p < 1/8$ as $M(k,\ell)\ge 3$ for $k\ge 2$.
Let $A$ be a random subset of $V$ with each vertex of $V$ chosen to be in $A$ independently at random with probability $p$.
Since $\Exp(|A|) = p |V|$, by Markov's inequality we have $\Pr(|A| > 2p |V|) \le 1/2$.
Let $\mathcal{F}$ be a smallest family of $k$-subsets of $[\ell]$ that does not have Property B.
We define a random $(k,\ell)$-list-assignment $L_A$ of $G[A]$ as follows: for each $v\in A$ independently, let $L_A(v)$ be a uniformly random element of $\mathcal{F}$.
We call a vertex $v$ {\em good} if $v\in V\setminus A$ and for any $F\in \mathcal{F}$ there are at least $tk+1$ neighbours $v'$ of $v$ such that $v'\in A$ and $L_A(v') = F$.

The probability that a vertex $v\in V\setminus A$ is not good is the probability that for some $F\in \mathcal{F}$ there are at most $tk$ neighbours $v'$ of $v$ such that $v'\in A$ and $L_A(v')=F$.  For each fixed $F\in\mathcal{F}$, since there are at least $D$ neighbours of $v$, the probability that there are at most $tk$ neighbours $v'$ of $v$ such that  $v'\in A$ and $L_A(v')=F$ must be at most
\begin{align*}
&\Pr\left(\Bin\left(D,\frac{p}{M(k,\ell)}\right)< tk + 1\right)\\
& \le
\Pr\left(\Bin\left(D,\frac{p}{M(k,\ell)}\right)-\frac{Dp}{M(k,\ell)}< -\frac{\sqrt{3D}}{M(k,\ell)}\sqrt{\frac{\ln M(k,\ell)}{\ln k}}\right)\\
&\le
2\exp(-4\ln M(k,\ell))
\end{align*}
due to Chernoff's bound (and some hidden routine algebraic manipulations).  Therefore, the probability that there is some set $F\in\mathcal{F}$ that certifies that $v$ is not good must be at most $M(k,\ell)\cdot 2\exp(-4\ln M(k,\ell))< 1/8$ (since $M(k,\ell)\ge 3$ for $k\ge 2$).

The probability that a vertex $v\in V$ is not good is at most the probability that $v\in A$, which is $p$, plus the probability computed above; thus, in total the probability is less than $p + 1/8 < 1/4$.  It then follows by Markov's inequality that the number of vertices that are not good is at most $|V|/2$ with probability strictly greater than $1/2$.  So there exists a set $A\subseteq V$ and a $(k,\ell)$-list-assignment $L_A$ such that $|A|\le 2p|V|$ and the number of good vertices is at least $|V|/2$.  Fix such a set $A$ and $(k,\ell)$-list-assignment $L_A$ for the remainder of the proof.  Let $A^*$ be the set of good vertices.

There are at most $k^{|A|}$ possibilities for an arbitrary $t$-improper $L_A$-colouring of $G[A]$.
We fix one such colouring $c_A$ and show that with high probability there is a $(k,\ell)$-list-assignment such that $c_A$ cannot be extended to a $t$-improper list colouring of $G[A\cup A^*]$.  We define a random $(k,\ell)$-list-assignment $L_{A^*}$ of $G[A^*]$ by letting, for each $v\in A^*$ independently, $L_{A^*}(v)$ be a uniformly random element of $\mathcal{F}$.
Consider a vertex $v\in A^*$.
We define $C_v$ to be the set of colours $c$ such that there are at least $t+1$ neighbours $v'$ of $v$ such that $v'\in A$ and $c_A(v')=c$.  We first observe that $C_v$ meets every member of $\mathcal{F}$.  Otherwise, let $F$ be some $k$-set in $\mathcal{F}$ that does not intersect $C_v$ and note that, since $v$ is good, $F$ occurs $tk+1$ times as the $L_A$-list of some neighbour of $v$.  But then, since $c_A$  is an $L_A$-colouring, it follows by the pigeonhole principle that there is some $f\in F$ such that $c_A(v')=f$ for at least $t+1$ vertices $v'$ of $A$, contradicting that $F$ does not meet $C_v$.  Now, since $\mathcal{F}$ satisfies Property B, there must be some $F_v \in \mathcal{F}$ such that $F_v\subseteq C_v$.  By the definition of $C_v$, if the random assignment results in $L_{A^*}(v) = F_v$, then $v$ cannot be $t$-improperly coloured from its list.  Since the choice of $L_{A^*}(v)$  is independent, the probability that all vertices of $A^*$ can be coloured from their lists is at most
\begin{align*}
\left(1-\frac{1}{M(k,\ell)}\right)^{|V|/2}<e^{-|V|/2M(k,\ell)}.
\end{align*} 
Note that
\begin{align*}
k^{|A|}e^{-|V|/2M(k,\ell)}
\le
k^{2p|V|}e^{-|V|/2M(k,\ell)}
= 1.
\end{align*}
It then follows that with positive probability there is a $(k,\ell)$-list-assignment $L$ of $G[A\cup A^*]$ such that there is no $t$-improper $L$-colouring of $G[A\cup A^*]$.
\end{proof}

\section{Proofs}
\label{sec:proofs}

\begin{proof}[Proof of Theorem~\ref{thm:completemultipartite}]
Combine Theorem~\ref{thm:multipartite,propB} and the bounds on $M(j,k)$ in~\eqref{eqn:propBj}.
\end{proof}

\begin{proof}[Proof of Theorem~\ref{thm:degeneracy}]
Combine Theorem~\ref{thm:mindeg} with the bounds on $M(k)$ in~\eqref{eqn:propB}.
\end{proof}

\begin{proof}[Proof of Theorem~\ref{thm:boundedspectrum}]
At the beginning of Section~\ref{sec:propB}, we observed that $M(k,2k-2)=\infty$ and $M(k,2k-1)=\binom{2k-1}{k}$.  The theorem follows by substituting these values into Proposition~\ref{prop:bipartite,propB} and Theorem~\ref{thm:mindeg}.
\end{proof}

\section{Conclusion}
\label{sec:conclusion}

We have shown that the correspondence between Property B and proper choosability is also present between Property B and improper choosability.  Using this association, we have shown new results for improper list colouring on complete multipartite graphs, graphs of bounded minimum degree and list assignments whose list union is bounded.  Our findings generalise classical results on proper list colouring.

Let us reiterate two open questions mentioned in Section~\ref{sec:results}:
\begin{enumerate}
\item For fixed $t\ge 1$, what is the fastest growing function $f:\mathbb{Z}^+\to\mathbb{Z}^+$ such that $\ch^t(G)\ge f(\ch(G))$ for all $G$?  We observed that $k \ge f(k)\ge \Omega(\ln k)$.  Could $f(k)$ be as high as $k/t$?
\item
For fixed $t\ge 1$, given a number $k$, does there exist a number $\ell_{t,k}$ such that a graph G is $t$-improperly $k$-choosable if it is $t$-improperly $(k,\ell_{t,k})$-choosable?
\end{enumerate}

Theorems~\ref{thm:completemultipartite} and~\ref{thm:degeneracy} suggest that the relaxation from proper to improper colouring does not significantly lower the corresponding choice numbers.  A similar phenomenon was also observed in the setting of improperly colouring random graphs, in which it was shown that even permitting $t$ to grow modestly in terms of the order of the random graph $G_{n,p}$ does not force the $t$-improper chromatic number (let alone the $t$-improper choice number) significantly lower than $n/(2\log_b n)$ a.a.s.~with $b = 1/(1-p)$~\cite{KaMc10}.  It would be interesting to determine how much we may allow $t$ to increase as a function of the order of the graph or the maximum degree before we discern a marked change in the asymptotic behaviour of the $t$-improper list chromatic numbers in the settings we have discussed.

\bibliographystyle{abbrv}
\bibliography{imchdegrees}

\end{document}